\documentclass{amsart}
\usepackage{tikz}
\usetikzlibrary{arrows.meta,decorations.pathreplacing}
\usepackage{xcolor}
\usepackage{amssymb,latexsym,amsmath,extarrows}
\usepackage{graphicx,mathrsfs}
\usepackage{hyperref,url}
\numberwithin{equation}{section}

\definecolor{capblue}{RGB}{42,113,155}
\definecolor{capteal}{RGB}{63,145,137}
\definecolor{caporange}{RGB}{191,104,62}
\tikzset{
  axis/.style={-{Stealth[length=4pt]},thin,gray},
  point/.style={circle,fill=black,inner sep=1.3pt},
  every picture/.style={font=\small}
}

\newtheorem{theorem}{Theorem}[section]
\newtheorem{lemma}[theorem]{Lemma}

\newtheorem{proposition}[theorem]{Proposition}
\newtheorem{remark}[theorem]{Remark}

\newtheorem{corollary}[theorem]{Corollary}

\newcommand{\e}{\epsilon}

\newcommand{\wh}{\widehat}

\newcommand{\ZR}{\mathbb{R}}
\newcommand{\R}{\mathbb{R}}
\newcommand{\ZT}{\mathbb{T}}

\newcommand{\ZZ}{\mathbb{Z}}

\newcommand{\supp}{{\rm supp}}

\begin{document}

\hypersetup{hidelinks}

\title{Diameter-free reverse inequalities and superorthogonality}

\author{Adam Cushman} \address{ Adam Cushman\\  Department of Mathematics\\ Indiana University Bloomington, USA} \email{acushma@iu.edu}

\author{Ciprian Demeter} \address{ Ciprian Demeter\\  Department of Mathematics\\ Indiana University Bloomington, USA} \email{demeterc@iu.edu}

\author{Shukun Wu} \address{ Shukun Wu\\  Department of Mathematics\\ Indiana University Bloomington, USA} \email{shukwu@iu.edu}

\begin{abstract}
We prove three results as part of the program of diameter-free estimates initiated in \cite{Cushman-Demeter-Wu}. 
The first two are reverse square function estimates for the light cone in $\mathbb{R}^3$. 
We first establish an abstract $L^4$ inequality of independent interest, under an ordered superorthogonality hypothesis: for every four distinct indices, only the two nonalternating pairings are required to vanish. 
The loss is $C(1+\log N)^2$, where $N$ is the number of functions.
An alternating-determinant argument verifies this hypothesis for separated cone sectors.
This gives a diameter-free estimate for
arbitrary disjoint angular intervals, at the thickness determined by their smallest width, with no upper restriction on the radial parameter. 

For the canonical equal-width partition on a fixed radial annulus, we obtain the loss $C(1+\log N)^{1/4}$, which is sharp up to constants. 
This refines the estimate by Guth-Wang-Zhang, via a different approach.
The improvement uses additional orthogonality between diagonal and off-diagonal differences, together with bounded overlap of dyadic difference shells. 

Our third result is the diameter-free $\ell^2L^6$ decoupling for arbitrary partitions of the parabola, with an $N^\varepsilon$ loss independent of the interval widths. 
The argument is an adaptation of the method from \cite{Cushman-Demeter-Wu} and implies their three-fold additive-energy estimate for the parabola. 

The three proofs use variants of interlacing and special orthogonality in place of wave packet analysis, multilinearity and parabolic/Lorentz rescaling.
Together, these results provide further evidence for the scope of the paradigm introduced in \cite{Cushman-Demeter-Wu}.
\end{abstract}

\maketitle

\section{Introduction}

The results in this paper are follow-ups to the diameter-free program initiated in \cite{Cushman-Demeter-Wu}. 
We present two reverse square function estimates for the cone and an $\ell^2L^6$ decoupling inequality for the parabola. 
The proofs are different, but have a common feature: the standard toolkit of wave packet analysis, multilinearity, and rescaling is replaced with variants of interlacing and special orthogonality. 
Their applicability to these three estimates gives further evidence for the breadth of this approach.

\smallskip 

The emergence of superorthogonality in this line of research is of independent interest. 
We therefore begin with an abstract result.

Let $r\geq1$ be an integer. Given a measure space $(X,\mu)$, a family
$\{f_j:j\in J\}$ of measurable functions is called superorthogonal for
$2r$-tuples of a certain type (see
\cite{Pierce-superorthogonality,GPRY}) if
\begin{equation}
\label{super-ortho}
    \int_X f_{j_1}\cdots f_{j_r}
    \overline{f_{j_{r+1}}\cdots f_{j_{2r}}}\,d\mu=0
\end{equation}
whenever $(j_1,\ldots,j_{2r})$ lies in a specified subset of $J^{2r}$.
The case $r=1$ coincides with standard orthogonality.
We work with finite families throughout the paper.

Typically, superorthogonality implies reverse square function estimates of the form
\begin{equation}
\nonumber
    \Big\|\sum_{j\in J}f_j\Big\|_p
    \ll \Big\|\Big(\sum_{j\in J}|f_j|^2\Big)^{1/2}\Big\|_p.
\end{equation}
Classic examples include Khintchine's inequality and the bi-orthogonality for the Fourier extension of convex curves.
We refer the reader to \cite{Pierce-superorthogonality,GPRY} for further discussion.

\smallskip

A particularly flexible condition is the type IV superorthogonality studied in \cite{GPRY}, which requires \eqref{super-ortho} whenever $j_1,\ldots,j_{2r}$ are distinct. 
For $r=2$, this still requires cancellation for all three pairings of four distinct indices into two unconjugated and two conjugated factors, up to conjugation, see the remark below. 
On the other hand, reverse square function estimates with a slowly growing dependence on $\#J$ may hold under weaker cancellation.
The cone estimate of Guth--Wang--Zhang \cite{GWZ}, whose proof uses wave packet geometry, provides a natural setting in which to investigate this question.

In this paper, we establish reverse square function estimates in $L^4$ under a weaker, ordered superorthogonality hypothesis for quadruples. 
We require only the two nonalternating pairings to vanish. 
The first result is the following abstract inequality.

\begin{samepage}
\begin{theorem}
\label{quartic-ortho-thm}
Let $N\geq1$ and $f_1,\ldots,f_N\in L^4(X,\mu)$.
Suppose that, for every $i<j<k<l$,
\begin{equation}
\label{ortho-hypo}
    \int_X f_i f_j\overline{f_k f_l}\,d\mu=0,
    \qquad
    \int_X f_i f_l\overline{f_j f_k}\,d\mu=0.
\end{equation}
Then there exists an absolute constant $C$ such that
\begin{equation}
\label{reverse-sqfcn}
    \Big\|\sum_{n=1}^N f_n\Big\|_4
    \leq C(1+\log N)^2
    \Big\|\Big(\sum_{n=1}^N|f_n|^2\Big)^{1/2}\Big\|_4.
\end{equation}
\end{theorem}
\end{samepage}

Some growth with $N$ in \eqref{reverse-sqfcn} is necessary, even under the additional superorthogonality used below. 
See Remark \ref{slow-growing-rmk}. 
We do not claim that the logarithmic exponent in Theorem~\ref{quartic-ortho-thm} is optimal.

\begin{remark}
\rm
For $i<j<k<l$, define
\begin{align*}
    A_{ijkl}&=\int_X f_i f_j\overline{f_k f_l}\,d\mu,
    &B_{ijkl}&=\int_X f_i f_l\overline{f_j f_k}\,d\mu,\\
    C_{ijkl}&=\int_X f_i f_k\overline{f_j f_l}\,d\mu.
\end{align*}The type IV superorthogonality studied in \cite{GPRY} required $A_{ijkl}=B_{ijkl}=C_{ijkl}=0$. On the other hand, 
Theorem~\ref{quartic-ortho-thm} says that if $A_{ijkl}=B_{ijkl}=0$, then the superficially weaker \eqref{reverse-sqfcn} holds, regardless of the behavior of $C_{ijkl}$.

The choice of ordered quadruples is important.
In fact, the vanishing of both $A_{ijkl}$ and $C_{ijkl}$ is not sufficient for an estimate such as \eqref{reverse-sqfcn}. 
Take $f_n(\theta)=e(n\theta)$ on $\ZT=\ZR/\ZZ$.
Then $A_{ijkl}=C_{ijkl}=0$, since $i+j<k+l$ and $i+k<j+l$.
However,
\[
    \Big\|\sum_{n=1}^N f_n\Big\|_4\asymp N^{3/4},
    \qquad
    \Big\|\Big(\sum_{n=1}^N|f_n|^2\Big)^{1/2}\Big\|_4=N^{1/2}.
\]
We do not know whether $B_{ijkl}=0$ alone suffices to prove
\eqref{reverse-sqfcn}.
\end{remark}

\bigskip  

\subsection{Applications to reverse square functions for the light cone}

Consider the parametrization of the light cone $r(1,t,t^2)$, with $r>0$ and
$t\in\ZR$. 

Let
\begin{equation}
\nonumber
    \Phi(r,t,u)=(r,rt,rt^2+u).
\end{equation}
We use the Fourier transform convention
\[
    \wh f(\xi)=\int_{\ZR^3}f(x)e(-x\cdot\xi)\,dx.
\]
For general $L^4$ functions, Fourier transforms and their supports are understood in the sense of tempered distributions.
As a direct corollary of Theorem~\ref{quartic-ortho-thm}, we obtain the following diameter-free reverse square function estimate.

\begin{samepage}
\begin{corollary}
\label{main-cor}
Let $N\geq1$ and consider ordered finite intervals
\begin{equation}
\nonumber
\begin{gathered}
    I_n=[\alpha_n,\beta_n],\qquad
    \alpha_n<\beta_n\leq\alpha_{n+1}\quad(1\leq n<N),\\
    \alpha_N<\beta_N.
\end{gathered}
\end{equation}
Set $w_n=\beta_n-\alpha_n$ and
\[
    w_*:=\min_{1\leq n\leq N}w_n>0.
\]
Let $a>0$, and assume that $g_1,\ldots,g_N\in L^4(\ZR^3)$ satisfy
\begin{equation}
\label{eq:radial-support}
    \supp(\wh g_n)\subset
    \{\Phi(r,t,u):r\geq a,\ t\in I_n,\ |u|\leq r w_*^2\}.
\end{equation}
Then, for an absolute constant $C$,
\begin{equation}
\label{eq:main-estimate}
    \Big\|\sum_{n=1}^N g_n\Big\|_4
    \leq C(1+\log N)^2
    \Big\|\Big(\sum_{n=1}^N|g_n|^2\Big)^{1/2}\Big\|_4.
\end{equation}
\end{corollary}
\end{samepage}

We call \eqref{eq:main-estimate} diameter-free because its constant does not depend on the widths of the intervals, their gaps, or their total span. 
It is also independent of $a$, and there is no upper restriction on $r$. 
The permitted thickness still depends on the smallest width through \eqref{eq:radial-support}.
Figure~\ref{fig:sectors} illustrates the geometry.

\begin{figure}[htbp]
\centering
\resizebox{.98\textwidth}{!}{%
\begin{tikzpicture}[x=1cm,y=1cm,>=Stealth,
  line cap=round,line join=round,
  figlabel/.style={font=\small},
  tinylab/.style={font=\footnotesize},
  maincurve/.style={thick,capblue}]
 \begin{scope}[shift={(2.8,0)}]
  \foreach \a/\b in {-1.04/-.79,-.53/-.19,.10/.22,.51/.96}{
    \path[fill=capblue!14]
      plot[domain=\a:\b,samples=20] ({1.05*2.55*\x},{2.55*(1+.23*\x*\x)})
      -- plot[domain=\b:\a,samples=20] ({1.05*.68*\x},{.68*(1+.23*\x*\x)}) -- cycle;
    \draw[capblue,thin] ({1.05*.68*\a},{.68*(1+.23*\a*\a)})
                         -- ({1.05*2.55*\a},{2.55*(1+.23*\a*\a)});
    \draw[capblue,thin] ({1.05*.68*\b},{.68*(1+.23*\b*\b)})
                         -- ({1.05*2.55*\b},{2.55*(1+.23*\b*\b)});
  }
  \foreach \r in {.68,1.6,2.55}{
    \draw[gray!65,thin] plot[domain=-1.12:1.12,samples=50]
      ({1.05*\r*\x},{\r*(1+.23*\x*\x)});
  }
  \foreach \t in {-1.12,-.79,-.19,.22,.96,1.12}{
    \draw[gray!60,densely dotted] (0,0)
      -- ({1.05*.68*\t},{.68*(1+.23*\t*\t)});
    \draw[gray!70,thin,->]
      ({1.05*2.55*\t},{2.55*(1+.23*\t*\t)})
      -- ({1.05*2.85*\t},{2.85*(1+.23*\t*\t)});
  }
  \fill (0,0) circle (1.4pt);
  \node[tinylab,below] at (0,0) {$0$};
  \node[tinylab,anchor=east] at (-.93,.61) {$r=a$};
  \node[figlabel] at (0,3.90) {ordered cone sectors};
  \node[tinylab] at (0,2.15) {$r(1,t,t^2)$};
 \end{scope}
 \begin{scope}[shift={(6.7,.9)}]
   \node[figlabel,anchor=west] at (0,2.1) {angular intervals};
   \draw[axis] (-.2,.9)--(5.1,.9) node[right] {$t$};
   \foreach \a/\b/\conefiglabel in {0/.92/1,1.29/2.47/2,2.88/3.30/3,3.71/4.87/4}{
     \fill[capblue!14] (\a,.83) rectangle (\b,.98);
     \draw[capblue,line width=1.5pt] (\a,.9)--(\b,.9);
     \draw[capblue] (\a,.77)--(\a,1.03) (\b,.77)--(\b,1.03);
     \node[tinylab] at ({(\a+\b)/2},1.30) {$I_{\conefiglabel}$};
   }
   \draw[decorate,decoration={brace,mirror,amplitude=4pt}]
      (2.88,.65)--(3.30,.65) node[midway,below=5pt,tinylab] {$w_*$};
   \node[tinylab,align=left,anchor=west] at (0,-.2)
     {unequal widths and arbitrary gaps\\$|u/r|\leq w_*^2$};
 \end{scope}
\end{tikzpicture}}
\caption{Ordered intervals determine sectors on the cone. 
The radial extent shown on the left is only illustrative: the diameter-free estimate has no upper radial cutoff. 
The error bound is imposed on the normalized height $u/r$.}
\label{fig:sectors}
\end{figure}
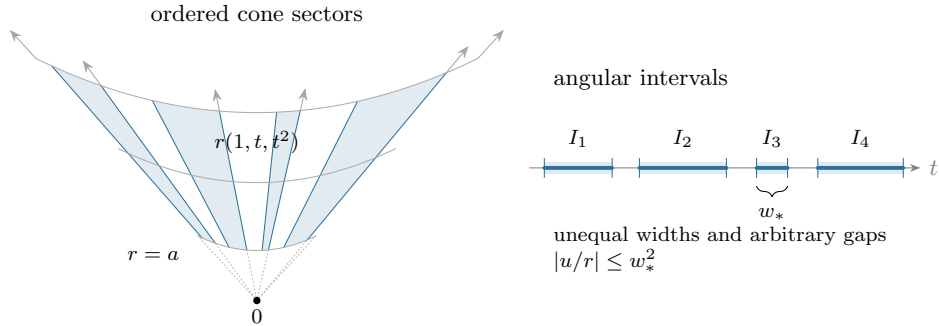

Our approach gives a sharper bound in the canonical setting.
For the canonical reverse square function estimate considered in \cite{GWZ}, the intervals form an equal-width partition, and the radial parameter is restricted to a fixed annulus.
In this setting, we prove the following estimate, which is sharp up to constants.

\begin{samepage}
\begin{theorem}
\label{thm-cone}
Let $N\geq2$, put $\delta=N^{-1}$, and take the consecutive partition
\begin{equation}
\label{eq:canonical-intervals}
    I_j=[(j-1)\delta,j\delta],\qquad 1\leq j\leq N.
\end{equation}
Define the conic sector
\[
    \Gamma_j=\{\Phi(r,t,u):1\leq r\leq2,\ t\in I_j,
    \ |u|\leq\delta^2\}.
\]
Then, for Schwartz functions $f_1,\ldots,f_N$ with
$\supp(\wh f_j)\subset\Gamma_j$,
\begin{equation}
\label{sq:eq:main}
    \Big\|\sum_{j=1}^N f_j\Big\|_4
    \leq C(1+\log N)^{1/4}
    \Big\|\Big(\sum_{j=1}^N|f_j|^2\Big)^{1/2}\Big\|_4,
\end{equation}
where $C$ is absolute.
\end{theorem}
\end{samepage}

In Subsection~\ref{sec:cone-sharpness}, we construct a nonzero Schwartz family with these Fourier supports such that
\begin{equation}
\nonumber
    \Big\|\sum_{j=1}^N f_j\Big\|_4
    \geq c(1+\log N)^{1/4}
    \Big\|\Big(\sum_{j=1}^N|f_j|^2\Big)^{1/2}\Big\|_4,
\end{equation}
which shows that Theorem~\ref{thm-cone} is sharp up to constants.

The proof of Theorem~\ref{thm-cone} follows a route similar to that of Corollary~\ref{main-cor}, with additional orthogonality properties that hold on the fixed radial annulus; see Subsection~\ref{sec:equal-cap-geometry}. 
The equal spacing of the centers in \eqref{eq:canonical-intervals} is also used in the proof.

\bigskip 

\subsection{Diameter-free decoupling inequality for the parabola}
\label{sec:parabola-statement}

Our third result gives another example of a diameter-free reverse inequality, which concerns an arbitrary partition of the parabola. This generalizes the canonical decoupling in \cite{Bourgain-Demeter}, providing an alternative proof. 
It uses the interlacing method of \cite{Cushman-Demeter-Wu}, together with orthogonality of products whose Fourier supports have controlled overlap and an elementary estimate for planar sectors.
It does not use the quartic superorthogonality theorem above.

Write $\gamma(t)=(t,t^2)$ for $0\leq t\leq1$, and let $\mathcal N_\rho(E)$ denote the Euclidean $\rho$-neighborhood of a set $E\subset\ZR^2$. 
We use the same Fourier transform convention in $\ZR^2$ as for the cone.

\begin{theorem}[Diameter-free $\ell^2L^6$ decoupling for the parabola]
\label{thm:parabola}
Fix $A\geq1$. Let $I_1<\cdots<I_N$ be a partition of $[0,1]$ into
nondegenerate intervals, and put
\[
    \ell=\min_{1\leq i\leq N}|I_i|.
\]
Suppose that $F_i\in\mathcal S(\ZR^2)$ and
\begin{equation}
\label{eq:parabola-support}
    \supp(\wh F_i)\subset\mathcal N_{A\ell^2}(\gamma(I_i)).
\end{equation}
Then, for every $\varepsilon>0$,
\begin{equation}
\label{eq:parabola-main}
    \Big\|\sum_{i=1}^N F_i\Big\|_{L^6(\ZR^2)}
    \leq C_{A,\varepsilon}N^\varepsilon
    \Big(\sum_{i=1}^N\|F_i\|_{L^6(\ZR^2)}^2\Big)^{1/2}.
\end{equation}
The constant is independent of the lengths and locations of the
intervals.
\end{theorem}

In particular, the loss in \eqref{eq:parabola-main} depends on the number of intervals, not on the reciprocal of the shortest length or on ratios of interval lengths. 
The shortest length determines only the allowed Fourier thickness in \eqref{eq:parabola-support}.
This answers the $p=6$ arbitrary-partition question in second author's work on decoupling for Cantor-like sets
\cite[Section~1]{Demeter-Cantor}, without a self-similarity assumption.

Theorem~\ref{thm:parabola} implies the parabola case of the three-fold additive-energy bound in \cite{Cushman-Demeter-Wu}.
This implication does not make the method independent of that work: the interlacing argument there provided both the methodology and the impetus for the decoupling  result presented here. 
There is a version of  Theorem~\ref{thm:parabola} for strictly convex curves. 
The proof is essentially the same, with a more careful choice of thickness in place of $Al^2$. 
In the interest of clarity, we decided not to include that argument.
We are not aware of any immediate application of either of these extensions of Theorem \ref{thm:parabola}, besides those explored in \cite{Cushman-Demeter-Wu}.

We give the proof in Section~\ref{sec:parabola}, followed by the weighted energy consequence in Subsection~\ref{sec:parabola-energy}.

\bigskip 

\subsection{Motivation and organization}

Our study of the diameter-free reverse inequality
\eqref{eq:main-estimate} was partially motivated by the diameter-free counting inequalities in \cite{Cushman-Demeter-Wu}. 
The cone argument uses an interlacing principle: the coefficients in the unique linear dependence among four ordered cone vectors alternate in sign. 
This excludes the two nonalternating pairings. 
The same property survives the thickening in \eqref{eq:radial-support} after separating the intervals. 
No counting theorem or parabola decoupling estimate is used as an input to either cone proof.

For the parabola, interlacing bounds the overlap of triple sums.
By Plancherel's theorem, this gives the special orthogonality needed for an induction on the number of intervals. 
Thus the three results share a method rather than a single abstract hypothesis: ordered cancellation for the cone and controlled Fourier overlap for the parabola. 
These distinct uses of interlacing illustrate how the program initiated in  \cite{Cushman-Demeter-Wu} extends beyond its original counting formulation.

\smallskip 

Section~\ref{proof-of-main-section} proves Theorem~\ref{quartic-ortho-thm}. 
A quartic identity controls an ordered imaginary-part sum, and a truncated Hilbert transform recovers the real part. 
Section~\ref{sec:geometry} establishes the cone geometry needed for both cone applications.
Section~\ref{sec:cone-estimates} contains the two cone proofs and the sharpness example. 
In the canonical argument, a dyadic decomposition of index gaps and bounded overlap of the corresponding Fourier supports give a single logarithm in the fourth power of the norm. 
Section~\ref{sec:parabola} contains the parabola proof and its energy consequence. 
None of the three upper bounds uses multilinear restriction, wave packet analysis, or parabolic rescaling.

\smallskip 

Throughout, $A\ll B$ means $A\leq CB$ for an absolute constant $C$, unless another dependence is stated; $A\gg B$ means $B\ll A$.
The notation $A\asymp B$ means that both inequalities hold.
We denote $e(t)=e^{2\pi\mathrm{i}t}$.
We use natural logarithms in estimates and $\log_2$ when counting dyadic scales.

\bigskip  

\subsection{AI usage} ChatGPT 5.6 Sol was instrumental in the production process of this manuscript. 
Once interlacing was discovered for curves in our previous work \cite{Cushman-Demeter-Wu}, we worked on its analogs in other contexts. 
One of these attempts led to the first manifestation of the forbidden pattern phenomenon for certain surfaces that resemble perturbed cones. 
The key superorthogonality emerged as the necessary tool to take advantage of the restricted number of permissible patterns for the cone.
ChatGPT Pro was also used to create pictures and to polish the paper.

\bigskip  

\noindent {\bf Acknowledgements.}
The second author is partially supported by the NSF grant DMS-2349828.
The third author is partially supported by the NSF grant DMS-2453583.

\bigskip 

\section{Proof of Theorem \ref{quartic-ortho-thm}}
\label{proof-of-main-section}

For simplicity, we set
\begin{equation*}
    S=\Big(\sum_{n=1}^N |f_n|^2\Big)^{1/2}.
\end{equation*}
If $\|S\|_4=0$, every $f_i$ vanishes almost everywhere, and there is nothing to prove for Theorem \ref{quartic-ortho-thm}.
Moreover, the case $N=1$ is trivial.

Now let us assume $\|S\|_4> 0$ and $N\geq2$.
Let $C(N)$ be the supremum of
\[
    \frac{\Big\|\sum_{n=1}^N f_n\Big\|_4}{\|S\|_4}
\]
over all families satisfying \eqref{ortho-hypo} with $\|S\|_4>0$.
Pointwise Cauchy--Schwarz gives $1\leq C(N)\leq\sqrt N$, so this
supremum is finite.  

We will repeatedly use the following observation:
both integrals in \eqref{ortho-hypo} remain zero and $S$ is unchanged
if $f_n$ is replaced by $c_n f_n$, where $|c_n|=1$.
Consequently, the same constant $C(N)$ applies to every such phase-changed family, and for
\[
    P_m(\theta)=\sum_{n=1}^m e(n\theta)f_n, \qquad 0\le m\le N,
\]
we have, for all $\theta\in[0,1]$,
\begin{equation}
\label{partial-sums}
    \|P_m(\theta)\|_4
    \le C(N)\Big\|\Big(\sum_{n=1}^m|f_n|^2\Big)^{1/2}\Big\|_4
    \le C(N)\|S\|_4.
\end{equation}
A shorter family is padded with zeros when applying $C(N)$.

\begin{samepage}
\begin{lemma}
Consider the truncated Hilbert kernel
\begin{equation}
\label{truncated-hilbert}
    K(\theta)=-2\sum_{n=1}^{N-1}\sin(2\pi n\theta).
\end{equation} 
Then, for every integer $n$ with $|n|<N$, we have
\begin{equation}
\label{hilbert}
    \int_0^{1}K(\theta)e(n\theta)\,d\theta=-\mathrm{i}\operatorname{sgn}(n),
\end{equation}
with $\operatorname{sgn}(0)=0$.
Moreover,
\[
    A_N=\int_0^{1}|K(\theta)|\,d\theta\ll1+\log N.
\]
\end{lemma}
\end{samepage}
\begin{proof}

The identity \eqref{hilbert} follows from orthogonality of $e(n\theta)$.
The geometric-series formula gives
\begin{equation*}
    |K(\theta)|
    \le 2\min\left\{N,\frac{1}{|\sin(\pi\theta)|}\right\}
    \qquad (\theta\notin\mathbb{Z}).
\end{equation*}
Using symmetry about $1/2$ and $|\sin(\pi\theta)|\geq2\theta$
on $[0,1/2]$, we obtain
\begin{equation}
\label{eq:oqo-kernel-bound}
    A_N\ll\int_0^{1/2}\min\{N,\theta^{-1}\}\,d\theta\ll1+\log N. \qedhere
\end{equation} 
    
\end{proof}

\begin{samepage}
\begin{lemma}
Let $N\geq2$ and let $f_1,\ldots,f_N\in L^4(X,\mu)$ satisfy
\eqref{ortho-hypo}. Put
\begin{equation*}
    a_{ij}=\operatorname{Im}(f_i\overline{f_j})\quad(i<j),
    \qquad
    V=V(f_1,\ldots, f_N)=2\sum_{i<j}a_{ij}.
\end{equation*}
Then, for coefficients $c_j$ with $|c_j|=1$, 
\begin{equation}
    \|V(c_1f_1,\ldots,c_N f_N)\|_2\leq 2A_N C(N)\|S\|_4^2.
\end{equation}

\end{lemma}
\end{samepage}
\begin{proof}

The cancellation hypothesis and $S$ are unchanged by phase changes, so we may assume $c_j=1$.
Expanding the square and grouping terms according to whether their index pairs meet, we have
\begin{equation}
\nonumber
\begin{aligned}
    \Big(\sum_{i<j}a_{ij}\Big)^2
    &=\sum_i\Big(\sum_{j<i}a_{ji}+\sum_{j>i}a_{ij}\Big)^2
    -\sum_{i<j}a_{ij}^2\\
    &\quad+2\sum_{i<j<k<l}
    \bigl(a_{ij}a_{kl}+a_{ik}a_{jl}+a_{il}a_{jk}\bigr).
\end{aligned}
\end{equation} 

The key observation is the identity 
\[
    2\operatorname{Im}z\operatorname{Im}w=\operatorname{Re}(z\overline{w}-zw),
\]
which implies that, for $i<j<k<l$,
\begin{equation*}
    a_{ij}a_{kl}+a_{ik}a_{jl}+a_{il}a_{jk}
    =\operatorname{Re}\left(
    f_i f_l\overline{f_j f_k}-f_i f_j\overline{f_k f_l}
    \right).
\end{equation*}
By \eqref{ortho-hypo}, its integral is $0$.
Thus, we have
\begin{equation*}
    \|V\|_2^2
    =4\int\sum_i
    \Big(\sum_{j<i}a_{ji}+\sum_{j>i}a_{ij}\Big)^2d\mu
    -4\sum_{i<j}\|a_{ij}\|_2^2.
\end{equation*}
Since
\begin{equation*}
    \sum_{j<i}a_{ji}+\sum_{j>i}a_{ij}
    =\operatorname{Im}\Big[
    \overline{f_i}\Big(\sum_{j<i}f_j-\sum_{j>i}f_j\Big)
    \Big],
\end{equation*}
we conclude that
\begin{equation}
\label{eq:oqo-row-energy}
    \|V\|_2^2
    \le 4\int\sum_i|f_i|^2
    \Big|\sum_{j<i}f_j-\sum_{j>i}f_j\Big|^2d\mu.
\end{equation}

To estimate this weighted sum, \eqref{hilbert} gives the exact representation
\begin{equation}
\label{eq:oqo-row-kernel}
    \sum_{j<i}f_j-\sum_{j>i}f_j
    =-\mathrm{i}\int_0^1K(\theta)e(-i\theta)P_N(\theta)\,d\theta.
\end{equation}
Indeed, the coefficient of $f_j$ on the right is
$\operatorname{sgn}(i-j)$. Minkowski's inequality, followed by
H\"older's inequality, gives
\begin{align*}
    &\Big(\int_X\sum_i|f_i|^2
      \Big|\sum_{j<i}f_j-\sum_{j>i}f_j\Big|^2\,d\mu\Big)^{1/2}\\
    &\quad\leq\int_0^1|K(\theta)|
      \Big(\int_X\sum_i|f_i|^2|P_N(\theta)|^2\,d\mu\Big)^{1/2}
      \,d\theta=\int_0^1|K(\theta)|\,
      \|S\cdot P_N(\theta)\|_2\,d\theta\\
    &\quad\leq\int_0^1|K(\theta)|\,
      \|S\|_4\|P_N(\theta)\|_4\,d\theta
      \leq A_N C(N)\|S\|_4^2,
\end{align*}
where the last line uses \eqref{partial-sums}.
Combining this with \eqref{eq:oqo-row-energy}, we have proved
\begin{equation}
\label{imaginary-1}
    \|V\|_2\le2A_N C(N)\|S\|_4^2.  \qedhere
\end{equation}

\end{proof}

\begin{samepage}
\begin{corollary}
\label{ortho-cor}
Suppose \eqref{ortho-hypo} holds and, in addition,
\begin{equation}
\label{another-super-ortho}
    \int_X |f_i|^2f_j\overline{f_k} \,d\mu=0
    \qquad\text{for every }i\text{ and }j\ne k.
\end{equation}
Then, for coefficients $c_j$ with $|c_j|=1$,
\begin{equation}
\label{stronger-conclusion}
    \|V(c_1f_1,\ldots,c_N f_N)\|_2\leq 2\|S\|_4^2.
\end{equation}
\end{corollary}
\end{samepage}

\begin{proof}
The additional hypothesis is also unchanged by phase changes, so we may take $c_j=1$. By \eqref{eq:oqo-row-energy} and \eqref{another-super-ortho}, all mixed terms in the squared row sum vanish upon integration. 
Hence
\begin{align*}
    \frac14\|V\|_2^2
    &\leq\int_X\sum_i|f_i|^2
       \Big|\sum_{j<i}f_j-\sum_{j>i}f_j\Big|^2\,d\mu\\
    &=\sum_i\sum_{j\ne i}\int_X|f_i|^2|f_j|^2\,d\mu
      \leq\int_X S^4\,d\mu.
\end{align*}
This gives \eqref{stronger-conclusion}.
\end{proof}

Corollary \ref{ortho-cor} will be useful to establish Theorem \ref{thm-cone}.

\bigskip 

\begin{proof}[Proof of Theorem \ref{quartic-ortho-thm}]

Note that
\[
    \Big|\sum_{n=1}^N f_n\Big|^2-S^2=2\sum_{i<j}\operatorname{Re}(f_i\overline{f_j}).
\]
To recover the real part from its imaginary part, for $\theta\in[0,1]$, define
\begin{equation*}
    V(\theta)=2\operatorname{Im}\sum_{i<j}
    e((i-j)\theta)f_i\overline{f_j}.
\end{equation*}
Applying \eqref{imaginary-1} to the phase-changed family $\{e(i\theta)f_i\}$ gives
\begin{equation*}
    \|V(\theta)\|_2\le2A_N C(N)\|S\|_4^2,
    \qquad \theta\in[0,1].
\end{equation*}
For $1\le n<N$ and $z\in\mathbb{C}$,
\eqref{hilbert} gives
\begin{equation*}
\int_0^{1}K(\theta)\,
    2\operatorname{Im}(e(-n\theta)z)\,d\theta
    =2\operatorname{Im}(\mathrm{i}z)=2\operatorname{Re}z.
\end{equation*}
Taking $n=j-i$ and $z=f_i\overline{f_j}$ and summing over $i<j$, we obtain
\begin{equation*}
    \Big|\sum_{n=1}^N f_n\Big|^2-S^2
    =\int_0^{1}K(\theta)V(\theta)\,d\theta.
\end{equation*}

Thus Minkowski's inequality gives
\begin{align*}
    \Big\|\Big|\sum_{n=1}^N f_n\Big|^2-S^2\Big\|_2
    \le\int_0^{1}|K(\theta)|\,
    \|V(\theta)\|_2\,d\theta\le2A_N^2 C(N)\|S\|_4^2,
\end{align*} 
which implies
\begin{equation*}
    \Big\|\sum_{n=1}^N f_n\Big\|_4^2
    \le\bigl(1+2A_N^2 C(N)\bigr)\|S\|_4^2.
\end{equation*}
Divide by $\|S\|_4^2$ and take the supremum over all admissible families. 
Since $C(N)$ is finite, this gives
\begin{equation*}
    C(N)^2\le1+2A_N^2 C(N).
\end{equation*}
Using $C(N)\ge1$ and \eqref{eq:oqo-kernel-bound}, we conclude that
\begin{equation*}
    C(N)\le C(N)^{-1}+2A_N^2
    \ll(1+\log N)^2.
\end{equation*}
Therefore
\begin{equation*}
    \Big\|\sum_{n=1}^N f_n\Big\|_4
    \ll(1+\log N)^2
    \Big\|\Big(\sum_{n=1}^N|f_n|^2\Big)^{1/2}\Big\|_4.
    \qedhere
\end{equation*}
\end{proof}

\bigskip 

\section{The geometry for the light cone}
\label{sec:geometry}

\subsection{Determinants for three vectors from the cone}

Recall \eqref{eq:radial-support}. 
For the support calculation, it is useful to keep the thickness parameter independent of the interval
widths. 
Fix $h>0$ and write
\[
    \xi=r(1,t,t^2+\e),\qquad
    \e=u/r,\qquad |\e|\leq h^2.
\]
For Corollary~\ref{main-cor}, we will take $h=w_*$.

The intervals carry their natural increasing order: we define $I_n\prec I_m$ if $\max I_n\leq\min I_m$, equivalently $n<m$.
Take three intervals $I_1\prec I_2\prec I_3$ with gaps of at least $2h$, and choose
\[
    \xi_j=r_j(1,t_j,t_j^2+\e_j),\qquad
    t_j\in I_j,\quad r_j>0,\quad |\e_j|\leq h^2.
\]
Set $d_1=t_2-t_1$ and $d_2=t_3-t_2$. Then
\begin{equation}
\label{eq:determinant}
\begin{aligned}
    D_3&:=\det\begin{pmatrix}
       1&t_1&t_1^2+\e_1\\
       1&t_2&t_2^2+\e_2\\
       1&t_3&t_3^2+\e_3
       \end{pmatrix}\\
    &=d_1d_2(d_1+d_2)+\e_1d_2-\e_2(d_1+d_2)+\e_3d_1.
\end{aligned}
\end{equation}
For a general bound $|\e_j|\leq\tau$, we have
\begin{equation}
\label{eq:minimum-determinant}
    \min_{|\e_j|\leq\tau}D_3=(d_1+d_2)(d_1d_2-2\tau).
\end{equation}
In our case, $d_1,d_2\geq2h$ and $\tau=h^2$. Hence
\[
    D_3\geq(d_1+d_2)(d_1d_2-2h^2)
    \geq2h^2(d_1+d_2)>0.
\]
Thus
\begin{equation}
\label{positivity}
    \det[\xi_1\ \xi_2\ \xi_3]=r_1r_2r_3D_3>0.
\end{equation}
Note that \eqref{positivity} is uniform in the radial parameters. 
No upper bound on the $r_j$ is required.

\medskip 

\subsection{Four ordered points: two forbidden pairings and one allowed pairing}

Given four intervals $I_1\prec I_2\prec I_3\prec I_4$ with gaps of at least $2h$, choose vectors $\xi_j$ from the corresponding thickened sectors given in \eqref{eq:radial-support}.
For example, the thickened sector of $I_1$ is $\{\Phi(r,t,u):r\geq a,\ t\in I_1,\ |u|\leq r w_*^2\}$.
For $s=1,2,3,4$, let $\Delta_s=\det[\xi_i\ \xi_j\ \xi_k]$, where $i<j<k$ are the three indices different from $s$. 
By \eqref{positivity}, every $\Delta_s$ is positive. 
Notice that, with $\xi_j=(x_j,y_j,z_j)$,
\[
    \det\begin{pmatrix}
      x_1&x_2&x_3&x_4\\
      x_1&x_2&x_3&x_4\\
      y_1&y_2&y_3&y_4\\
      z_1&z_2&z_3&z_4
    \end{pmatrix}
    =x_1\Delta_1-x_2\Delta_2+x_3\Delta_3-x_4\Delta_4=0.
\]
The same holds when the first row is replaced by
$(y_1,y_2,y_3,y_4)$ or $(z_1,z_2,z_3,z_4)$. Therefore,
\begin{equation}
\label{alternating}
    \Delta_1\xi_1-\Delta_2\xi_2+\Delta_3\xi_3-\Delta_4\xi_4=0.
\end{equation}

By \eqref{positivity}, the matrix $[\xi_1\ \xi_2\ \xi_3\ \xi_4]$ has rank three, so its kernel is one-dimensional and is spanned by the coefficient vector in \eqref{alternating}. 
Thus every nonzero real relation $\sum_j c_j\xi_j=0$ has alternating coefficient signs.
In particular,
\[
    \xi_1+\xi_2=\xi_3+\xi_4,
    \qquad
    \xi_1+\xi_4=\xi_2+\xi_3
\]
are impossible: their signs are $(+,+,-,-)$ and $(+,-,-,+)$.
The alternating relation $\xi_1+\xi_3=\xi_2+\xi_4$ is not excluded.
See Figure~\ref{fig:pairings}.

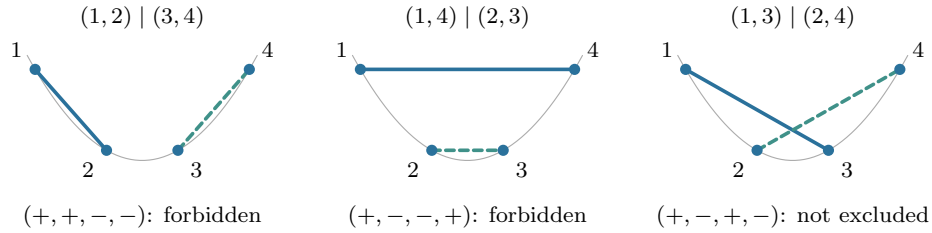
\begin{figure}[htbp]
\centering
\resizebox{.98\textwidth}{!}{%
\begin{tikzpicture}[x=1cm,y=1cm,>=Stealth,
  line cap=round,line join=round,
  figlabel/.style={font=\small},
  tinylab/.style={font=\footnotesize},
  maincurve/.style={thick,capblue}]
\foreach \xshift/\kind in {0/1,4.1/2,8.2/3}{
 \begin{scope}[shift={(\xshift,0)}]
  \draw[gray!65,thin] plot[domain=-.1:2.8,samples=50]
     (\x,{.25+.63*(\x-1.35)*(\x-1.35)});
  \coordinate (q1) at (0,1.398);
  \coordinate (q2) at (.90,.378);
  \coordinate (q3) at (1.80,.378);
  \coordinate (q4) at (2.70,1.398);
  \ifnum\kind=1
    \draw[capblue,line width=1.3pt] (q1)--(q2);
    \draw[capteal,line width=1.3pt,densely dashed] (q3)--(q4);
    \node[tinylab] at (1.35,2.05) {$(1,2)\mid(3,4)$};
    \node[tinylab] at (1.35,-.45) {$(+,+,-,-)$: forbidden};
  \fi
  \ifnum\kind=2
    \draw[capblue,line width=1.3pt] (q1)--(q4);
    \draw[capteal,line width=1.3pt,densely dashed] (q2)--(q3);
    \node[tinylab] at (1.35,2.05) {$(1,4)\mid(2,3)$};
    \node[tinylab] at (1.35,-.45) {$(+,-,-,+)$: forbidden};
  \fi
  \ifnum\kind=3
    \draw[capblue,line width=1.3pt] (q1)--(q3);
    \draw[capteal,line width=1.3pt,densely dashed] (q2)--(q4);
    \node[tinylab] at (1.35,2.05) {$(1,3)\mid(2,4)$};
    \node[tinylab] at (1.35,-.45) {$(+,-,+,-)$: not excluded};
  \fi
  \foreach \j in {1,2,3,4}{\fill[capblue] (q\j) circle (2pt);}
  \node[tinylab,above left=1pt] at (q1) {$1$};
  \node[tinylab,below left=1pt] at (q2) {$2$};
  \node[tinylab,below right=1pt] at (q3) {$3$};
  \node[tinylab,above right=1pt] at (q4) {$4$};
 \end{scope}
}
\end{tikzpicture}}
\caption{The three pairings in the normalized cone section
$(1,t,t^2)$. A positive two-against-two cone relation forces the two
corresponding chords in the $(t,t^2)$-plane to meet. Only the
alternating chords can do so. The determinant calculation proves
that the same obstruction survives the allowed thickening after
separating the intervals.}
\label{fig:pairings}
\end{figure}

The above discussion implies the following lemma.
\begin{samepage}
\begin{lemma}
\label{support-lem}
Let $a,h>0$, and let $I_1\prec\cdots\prec I_N$ have gaps of at
least $2h$. Suppose that $g_n\in L^4(\ZR^3)$ and
\[
    \supp(\wh g_n)\subset
    \{\Phi(r,t,u):r\geq a,\ t\in I_n,\ |u|\leq rh^2\}.
\]
Then, for every $i<j<k<l$,
\[
    \int_{\ZR^3}g_i g_j\overline{g_k g_l}\,dx=0,
    \qquad
    \int_{\ZR^3}g_i g_l\overline{g_j g_k}\,dx=0.
\]
\end{lemma}
\end{samepage}

\begin{proof}
The forbidden relations show that the Fourier supports of $g_i g_j$ and $g_k g_l$ are disjoint, and likewise for $g_i g_l$ and $g_j g_k$.
Plancherel's theorem gives the conclusion.
To justify this for general $L^4$ functions, first multiply each $\wh g_n$ by $\chi(\xi/R)$, where $\chi\in C_c^\infty(\ZR^3)$ equals $1$ near the origin. 
This does not enlarge Fourier support.
The cutoff functions have compact Fourier supports, and the supports of their products lie in the corresponding compact Minkowski sums. 
The cutoff functions converge to $g_n$ in $L^4$ as $R\longrightarrow\infty$, and their products converge in $L^2$.
Passing to the limit proves the stated formulation.
\end{proof}

\medskip 

\subsection{Additional geometry for equal-size caps}
\label{sec:equal-cap-geometry}

Let $I_1\prec\cdots\prec I_M$ be intervals of the same length $w>0$ which are $\geq10w$ apart.
Write
\begin{equation}
\label{Ij-Omegaj}
    I_j=[t_j-w/2,t_j+w/2],\quad
    \Omega_j=\{\Phi(r,t,u):1\leq r\leq2,\ t\in I_j,
    \ |u|\leq w^2\}.
\end{equation}
In particular, $t_{j+1}-t_j\geq11w$.
The centers $\{t_j\}$ need not be equally spaced in this subsection. 
Define
\[
    Q(\xi)=\xi_1\xi_3-\xi_2^2.
\]
For $\xi=\Phi(r,t,u)$ and $\eta=\Phi(r',t',u')$, direct calculation gives
\begin{equation}
\label{cone-q-identity}
    Q(\xi-\eta)=-rr'(t-t')^2+(r-r')(u-u').
\end{equation}

\begin{samepage}
\begin{lemma}
\label{cone-difference-lem}
Let $I_j$ and $\Omega_j$ be given in \eqref{Ij-Omegaj}.
For every $1\leq i\leq M$, we have
\begin{equation}
\label{cone-same-cap-difference}
    \Omega_i-\Omega_i\subset\{\xi:|Q(\xi)|\leq6w^2\}.
\end{equation}
For $1\leq i,j\leq M$ with $i\ne j$, we have
\begin{equation}
\label{cone-difference-comparison}
    \Omega_i-\Omega_j\subset
    \Big\{\xi:\frac12|t_i-t_j|^2\leq -Q(\xi)
    \leq5|t_i-t_j|^2\Big\}.
\end{equation}
Moreover,
\begin{equation}
\label{cone-repeated-separation}
    (\Omega_i-\Omega_i)\cap(\Omega_j-\Omega_k)=\varnothing
    \qquad(1\leq i,j,k\leq M,\ j\ne k).
\end{equation}
\end{lemma}
\end{samepage}

\begin{proof}
Since $1\leq r,r'\leq2$ and $|u|,|u'|\leq w^2$, we have
\[
    1\leq rr'\leq4,\qquad |(r-r')(u-u')|\leq2w^2.
\]
If $\xi,\eta\in\Omega_i$, then $|t-t'|\leq w$.
Hence \eqref{cone-q-identity} gives
\[
    |Q(\xi-\eta)|\leq4w^2+2w^2=6w^2,
\]
which proves \eqref{cone-same-cap-difference}.

If $\xi\in\Omega_j$ and $\eta\in\Omega_k$ with $j\ne k$, then $|t-t'|\geq10w$. 
Therefore,
\[
    -Q(\xi-\eta)\geq(t-t')^2-2w^2\geq98w^2.
\]
Since $98>6$, this proves \eqref{cone-repeated-separation}, including
when $i=j$ or $i=k$.

Finally, suppose $\xi\in\Omega_i$, $\eta\in\Omega_j$, and $i\ne j$.
Put $d=|t_i-t_j|\geq11w$. Since
\[
    d-w\leq|t-t'|\leq d+w,
\]
we obtain
\begin{align*}
    -Q(\xi-\eta)&\geq(d-w)^2-2w^2\geq\frac12d^2,\\
    -Q(\xi-\eta)&\leq4(d+w)^2+2w^2\leq5d^2.
\end{align*}
The last inequalities use $w\leq d/11$.
This proves \eqref{cone-difference-comparison}.
\end{proof} 

\begin{samepage}
\begin{lemma}
\label{cone-scale-geometry}
Let $I_1\prec\cdots\prec I_M$ be intervals of the same length $w>0$ which are $\geq10w$ apart.
Let $\Omega_j$ be given in \eqref{Ij-Omegaj}.
Fix $D>0$. 
For integers $s\geq0$, let
\begin{equation}
\label{cone-gap-regions}
    \Sigma_s=
    \bigcup_{\substack{1\leq i<j\leq M\\
        2^{s-1}D\leq t_j-t_i\leq2^{s+1}D}}
    (\Omega_i-\Omega_j).
\end{equation}
Then $\Sigma_s\cap\Sigma_{s'}=\varnothing$ whenever $|s-s'|\geq4$.
Consequently,
\begin{equation}
\label{cone-gap-overlap}
    \sup_{\xi\in\ZR^3}\sum_{s\geq0}\mathbf{1}_{\Sigma_s}(\xi)
    \leq4.
\end{equation}
\end{lemma}
\end{samepage}

\begin{proof}
By \eqref{cone-difference-comparison},
\begin{equation}
\nonumber
    \Sigma_s\subset E_s
    :=\Big\{\xi:\frac18D^2 4^s
    \leq -Q(\xi)\leq20D^2 4^s\Big\}.
\end{equation}
If $s'\geq s+4$, then the lower endpoint of the interval defining $E_{s'}$ is at least $32D^2 4^s$, which exceeds the upper endpoint $20D^2 4^s$ for $E_s$. 
Thus $E_s\cap E_{s'}=\varnothing$. 
This proves \eqref{cone-gap-overlap}.
\end{proof}

\bigskip 

\section{Proof of the reverse square function estimates for the cone}
\label{sec:cone-estimates}

\subsection{Proof of Corollary \ref{main-cor}}
\label{sec:finish}

We now combine the geometry in Section~\ref{sec:geometry} with Theorem~\ref{quartic-ortho-thm}. 
Partition the original indices modulo $3$. 
Two successive intervals in one class have two intervening intervals, each of width at least $w_*$.
Hence their gap is at least $2w_*$. 
Lemma~\ref{support-lem}, with $h=w_*$, verifies
\eqref{ortho-hypo} in each class.

For the three residue classes $\mathcal R_c$, put
\[
    G_c=\sum_{n\in\mathcal R_c}g_n,
    \qquad
    S_c=\Big(\sum_{n\in\mathcal R_c}|g_n|^2\Big)^{1/2}.
\]
Theorem~\ref{quartic-ortho-thm} and the triangle inequality give
\begin{align*}
    \Big\|\sum_{n=1}^N g_n\Big\|_4
    &\leq\sum_{c=1}^3\|G_c\|_4
      \ll(1+\log N)^2\sum_{c=1}^3\|S_c\|_4\\
    &\ll(1+\log N)^2
      \Big\|\Big(\sum_{n=1}^N|g_n|^2\Big)^{1/2}\Big\|_4.
\end{align*}
This proves Corollary~\ref{main-cor}.

\medskip 

\subsection{Auxiliary lemmas for equal-size caps}
\label{sec:cone-auxiliary}

Consider a separated family of $M$ intervals and caps as in \eqref{Ij-Omegaj}, and let $f_1,\ldots,f_M$ be Schwartz functions whose Fourier transforms are supported in these caps. 
Put
\[
    S=\Big(\sum_{j=1}^M|f_j|^2\Big)^{1/2},
    \qquad
    U=\sum_{1\leq i<j\leq M}f_i\overline{f_j},
\]
and let $L=1+\lceil\log_2 M\rceil$. 
Thus $L$ is an integer and $L\ll1+\log M$.

The next lemma is a standard partition of unity for the positive frequencies of the truncated Hilbert transform \eqref{truncated-hilbert}.
We include its proof to record that the $L^1$ bound is independent of the scale.

\begin{samepage}
\begin{lemma}
\label{cone-kernel-lem}
There exists a nonnegative $\psi\in C_c^\infty(\ZR)$ supported in $[1/2,2]$ such that
\begin{equation}
\label{cone-gap-partition}
    \sum_{s=0}^{L-1}\psi(2^{-s}k)=1
    \qquad(1\leq k\leq M,\ k\in\ZZ).
\end{equation}
For every integer $s\geq0$, let
\begin{equation}
\nonumber
    K_s(\theta)=\sum_{n\in\ZZ}\psi(2^{-s}n)e(n\theta).
\end{equation}
Then, independently of $s$,
\begin{equation}
\nonumber
    \int_0^1|K_s(\theta)|\,d\theta\ll1.
\end{equation}
\end{lemma}
\end{samepage}

\begin{proof}
Choose a smooth even cutoff $\eta$ that equals $1$ on $[-1,1]$, vanishes outside $[-2,2]$, and is nonincreasing on $[0,\infty)$.
For $x>0$, put $\psi(x)=\eta(x)-\eta(2x)$, and put $\psi(x)=0$ for $x\leq0$. 
Then $\psi$ is smooth, nonnegative, and supported in
$[1/2,2]$. 
The sum in \eqref{cone-gap-partition} telescopes to
\[
    \eta(2^{-(L-1)}k)-\eta(2k)=1.
\]
Writing $\check\psi(y)=\int_{\ZR}\psi(v)e(vy)\,dv$, Poisson summation gives
\[
    K_s(\theta)=2^s\sum_{m\in\ZZ}
      \check\psi\big(2^s(\theta+m)\big).
\]
Consequently,
\[
    \int_0^1|K_s(\theta)|\,d\theta
    \leq\sum_{m\in\ZZ}\int_0^1
      2^s\big|\check\psi\big(2^s(\theta+m)\big)\big|\,d\theta=\|\check\psi\|_{L^1(\ZR)}<\infty.  \qedhere
\]
\end{proof}

\medskip 

\begin{lemma}
\label{cone-dyadic-lem}
Let $I_1\prec\cdots\prec I_M$ be intervals of the same length $w>0$ which are $\geq10w$ apart, and recall \eqref{Ij-Omegaj}.
Take $\psi$ as in Lemma \ref{cone-kernel-lem}.
For integers $0\leq s<L$, define
\begin{equation}
\nonumber
    U_s=\sum_{1\leq i<j\leq M}
      \psi\Big(\frac{j-i}{2^s}\Big)f_i\overline{f_j}.
\end{equation}
Then $U=\sum_{s=0}^{L-1}U_s$ and
\begin{equation}
\nonumber
    \|U_s\|_2\ll\|S\|_4^2.
\end{equation}
\end{lemma}

\begin{proof}
First, we verify the hypotheses of Corollary~\ref{ortho-cor}.
Since the intervals $\{I_j\}$ are separated by at least $10w$ and $|u/r|\leq w^2$, Lemma~\ref{support-lem}, with $h=w$ and $a=1$, gives \eqref{ortho-hypo}. 
Moreover,
\[
    \supp\big(\wh{|f_i|^2}\big)\subset\Omega_i-\Omega_i,
    \qquad
    \supp\big(\wh{f_k\overline{f_j}}\big)
      \subset\Omega_k-\Omega_j.
\]
By \eqref{cone-repeated-separation} and Plancherel's theorem,
\[
    \int_{\ZR^3}|f_i|^2f_j\overline{f_k}\,dx
    =\langle|f_i|^2,f_k\overline{f_j}\rangle_{L^2}=0
    \qquad(j\ne k).
\]
This includes $i=j$ and $i=k$, and proves
\eqref{another-super-ortho}.

For $\theta\in[0,1]$, define
\[
    V(\theta)=2\operatorname{Im}\sum_{1\leq i<j\leq M}
      e((i-j)\theta)f_i\overline{f_j}.
\]
Corollary~\ref{ortho-cor}, applied to the phase-changed family
$\{e(j\theta)f_j\}$, gives
\[
    \|V(\theta)\|_2\leq2\|S\|_4^2
    \qquad(\theta\in[0,1]).
\]
The positive frequencies in $K_s$ give the exact identity
\begin{equation}
\label{cone-phase-recovery}
    U_s=\int_0^1K_s(\theta)\,\mathrm{i}V(\theta)\,d\theta.
\end{equation}
Indeed, a pair $i<j$ contributes
\[
    e((i-j)\theta)f_i\overline{f_j}
    -e((j-i)\theta)\overline{f_i}f_j
\]
to $\mathrm{i}V(\theta)$. 
After integration against $K_s$, the first coefficient is $\psi(2^{-s}(j-i))$, while the second is
$-\psi(2^{-s}(i-j))=0$. 
Therefore, Minkowski's inequality gives
\begin{align*}
    \|U_s\|_2
    &\leq\int_0^1|K_s(\theta)|\,\|V(\theta)\|_2\,d\theta\\
    &\leq2\|S\|_4^2\int_0^1|K_s(\theta)|\,d\theta
      \ll\|S\|_4^2.
\end{align*}
Finally, $U=\sum_s U_s$ follows from \eqref{cone-gap-partition}.
\end{proof}

\medskip 

\subsection{Proof of Theorem \ref{thm-cone}}
\label{sec:cone-proof}

As in Subsection~\ref{sec:finish}, partition the original indices modulo $11$. 
The intervals in each class have gaps of at least $10\delta$. 
Fix one nonempty class and relabel its $M$ intervals in increasing order. 
Their centers satisfy
\begin{equation}
\label{cone-center-spacing}
    t_j-t_i=11\delta(j-i).
\end{equation}
Take $L=1+\lceil\log_2 N\rceil$. 

Apply Lemma~\ref{cone-scale-geometry} to this $M$-cap family  with $D=11\delta$.
The support of $\psi$ and \eqref{cone-center-spacing} give
\[
    \supp(\wh U_s)\subset\Sigma_s,
\]
where $\Sigma_s$ is defined in \eqref{cone-gap-regions}.
The bounded-overlap property \eqref{cone-gap-overlap} provides additional orthogonality for the $U_s$.
By Cauchy--Schwarz at each frequency and Plancherel's theorem,
\begin{align*}
    \|U\|_2^2
    &=\int_{\ZR^3}\Big|\sum_{s=0}^{L-1}\wh U_s(\xi)\Big|^2\,d\xi\\
    &\leq\int_{\ZR^3}
      \Big(\sum_{s=0}^{L-1}\mathbf{1}_{\Sigma_s}(\xi)\Big)
      \sum_{s=0}^{L-1}|\wh U_s(\xi)|^2\,d\xi\\
    &\leq4\sum_{s=0}^{L-1}\|U_s\|_2^2
      \ll L\|S\|_4^4.
\end{align*}
The last inequality uses Lemma~\ref{cone-dyadic-lem}. Since
\[
    \Big|\sum_{j=1}^M f_j\Big|^2=S^2+2\operatorname{Re}U,
\]
we obtain
\begin{equation}
\nonumber
\begin{aligned}
    \Big\|\sum_{j=1}^M f_j\Big\|_4^4 \leq2\|S\|_4^4+8\|U\|_2^2\ll(1+\log N)\|S\|_4^4.
\end{aligned}
\end{equation}
Taking fourth roots and recombining the at most $11$ residue classes, as in Subsection~\ref{sec:finish}, proves \eqref{sq:eq:main}. 

\medskip 

\subsection{Sharpness of Theorem \ref{thm-cone}}
\label{sec:cone-sharpness}
The following proposition implies that \eqref{sq:eq:main} is sharp up to constants.

\begin{samepage}
\begin{proposition}
\label{cone-sharpness-prop}
For every $N\geq2$, there is a nonzero Schwartz family
$f_1,\ldots,f_N$ with $\supp(\wh f_j)\subset\Gamma_j$ such that
\begin{equation}
\label{cone-sharpness-conclusion}
    \Big\|\sum_{j=1}^Nf_j\Big\|_4
    \geq c(1+\log N)^{1/4}
    \Big\|\Big(\sum_{j=1}^N|f_j|^2\Big)^{1/2}\Big\|_4,
\end{equation}
where $c>0$ is absolute.
\end{proposition}
\end{samepage}

\begin{proof}
Let $\delta=N^{-1}$ and $t_j=(j-\frac12)\delta$.
Choose a fixed nonnegative function
\[
    \varphi\in C_c^\infty\big((1,2)\times(-1/100,1/100)^2\big),
    \qquad \int\varphi=1,
\]
and define
\begin{equation}
\label{cone-example-bumps}
    \wh f_j(\xi)=
    \varphi\Big(
        \xi_1,\frac{\xi_2-t_j\xi_1}{\delta},
        \frac{\xi_3-2t_j\xi_2+t_j^2\xi_1}{\delta^2}
    \Big),
\end{equation}
a smooth bump on a cap-adapted frequency box.
On this support, writing $\xi=\Phi(r,t,u)$, we have $1<r<2$ and
\[
    |t-t_j|\leq\frac{\delta}{100}.
\]
Also,
\[
    u=\xi_3-\frac{\xi_2^2}{\xi_1}
    =\xi_3-2t_j\xi_2+t_j^2\xi_1
    -\frac{(\xi_2-t_j\xi_1)^2}{\xi_1},
\]
so
\begin{equation}
\nonumber
    |u|
    \leq\Big(\frac1{100}+\frac1{10000}\Big)\delta^2
    <\frac{\delta^2}{50}.
\end{equation}
Thus, $\supp(\wh f_j)\subset\Gamma_j$, and every $f_j$ is Schwartz.
The linear change of variables in \eqref{cone-example-bumps} has
inverse Jacobian $\delta^3$. Hence
\begin{equation}
\label{cone-example-mass}
    \wh f_j\geq0,\qquad
    \|\wh f_j\|_1=\delta^3,\qquad
    \|\wh f_j\|_\infty\ll1.
\end{equation}
All implicit constants below depend only on the fixed function
$\varphi$.

\medskip

We first estimate the square function.
For $i\ne j$ and for fixed $\xi\in\ZR^3$, the overlap
$\supp(\wh f_i)\cap(\xi-\supp(\wh f_j))$ is contained in the intersection of the following three slabs in the variable $\eta\in\ZR^3$:
\[
      |\eta_2-t_i\eta_1| \leq\frac{\delta}{100},
\]
\[
    |\eta_3-2t_i\eta_2+t_i^2\eta_1|
    \leq\frac{\delta^2}{100},
\]
\[
    \big|(\xi_3-\eta_3)-2t_j(\xi_2-\eta_2)
    +t_j^2(\xi_1-\eta_1)\big|
    \leq\frac{\delta^2}{100}.
\]
These come from the second and third coordinates defining $\wh f_i$ in \eqref{cone-example-bumps}, and the third coordinate defining $\wh f_j(\xi-\cdot)$.
The absolute determinant of their normal vectors is
\[
    \left|
    \det\begin{pmatrix}
        -t_i&1&0\\
        t_i^2&-2t_i&1\\
        t_j^2&-2t_j&1
    \end{pmatrix}
    \right|
    =(t_i-t_j)^2=\delta^2(i-j)^2.
\]
Therefore, the overlap has volume
$\ll\delta^5/(t_i-t_j)^2$, and \eqref{cone-example-mass} gives
\[
    \|\wh f_i*\wh f_j\|_\infty
    \ll\frac{\delta^5}{(t_i-t_j)^2}
    =\frac{\delta^3}{(i-j)^2}.
\]
For $i=j$, we use
\[
    \|\wh f_i*\wh f_i\|_\infty
    \leq\|\wh f_i\|_\infty\|\wh f_i\|_1
    \ll\delta^3.
\]
Moreover, nonnegativity gives
\[
    \|\wh f_i*\wh f_j\|_1
    =\|\wh f_i\|_1\|\wh f_j\|_1
    =\delta^6.
\]
Hence Plancherel's theorem implies
\begin{equation}
\label{cone-example-denominator}
\begin{aligned}
    \Big\|\Big(\sum_{j=1}^N|f_j|^2\Big)^{1/2}\Big\|_4^4
    &=\sum_{i,j=1}^N\|f_if_j\|_2^2
    =\sum_{i,j=1}^N\|\wh f_i*\wh f_j\|_2^2\\
    &\leq\sum_{i,j=1}^N
    \|\wh f_i*\wh f_j\|_\infty
    \|\wh f_i*\wh f_j\|_1\\
    &\ll\delta^9
    \sum_{i,j=1}^N\frac1{(1+|i-j|)^2}
    \ll N\delta^9=\delta^8.
\end{aligned}
\end{equation}

\medskip

We next obtain a lower bound for the fourth moment.
For each dyadic integer $K$ with $1\leq K\leq N/2$, put
\[
    H_K=
    \sum_{\substack{1\leq i<j\leq N\\K\leq j-i<2K}}
    \wh f_i*\wh f_j.
\]
All sums over $K$ below are restricted to these dyadic integers.
By \eqref{cone-example-mass}, every convolution is nonnegative
and has integral $\delta^6$. Therefore,
\begin{equation}
\label{cone-example-dyadic-mass}
\begin{aligned}
    \|H_K\|_1=\delta^6\sum_{k=K}^{2K-1}(N-k)=\delta^6\Big(NK-\frac{K(3K-1)}2\Big)
      \geq\frac14NK\delta^6.
\end{aligned}
\end{equation}
The last inequality uses $K\leq N/2$.
Next, we show that $H_K$ is supported in a neighborhood of the
cone of thickness $O((K\delta)^2)$ on a fixed radial region.

Recall that $Q(\xi)=\xi_1\xi_3-\xi_2^2$.
Direct calculation gives
\begin{equation}
\nonumber
    Q\big(\Phi(r,t,u)+\Phi(r',t',u')\big)
    =rr'(t-t')^2+(r+r')(u+u').
\end{equation}
For a pair contributing to $H_K$, the angular parameters satisfy
\[
    |t-t'|
    \leq\delta\Big(j-i+\frac1{50}\Big)
    \leq\delta\Big(2K-1+\frac1{50}\Big)
    <2K\delta.
\]
Since $1<r,r'<2$ and $|u|,|u'|<\delta^2/50$, it follows that
\[
    \supp(H_K)\subset
    \left\{\xi:
        2<\xi_1<4,\quad 0<\xi_2<4,\quad
        |Q(\xi)|\leq20(K\delta)^2
    \right\}.
\]
For fixed $(\xi_1,\xi_2)$, the last condition is equivalent to
\[
    \frac{\xi_2^2-20(K\delta)^2}{\xi_1}
    \leq\xi_3\leq
    \frac{\xi_2^2+20(K\delta)^2}{\xi_1}.
\]
This interval has length at most $20(K\delta)^2$, since $\xi_1\geq2$.
The first two coordinates range over a rectangle of area $8$.
Consequently,
\begin{equation}
\label{cone-example-dyadic-volume}
    |\supp(H_K)|\leq160(K\delta)^2.
\end{equation}

By Cauchy--Schwarz, \eqref{cone-example-dyadic-mass}, and
\eqref{cone-example-dyadic-volume},
\begin{equation}
\label{cone-example-one-scale}
    \|H_K\|_2^2
    \geq\frac{\|H_K\|_1^2}{|\supp(H_K)|}
    \gg\frac{(NK\delta^6)^2}{(K\delta)^2}
    =N^2\delta^{10}
    =\delta^8.
\end{equation}
Thus every dyadic gap size gives the same lower bound.
The ranges $[K,2K)$ are disjoint as sets of index gaps, so each pair $i<j$ occurs in at most one $H_K$. 
By nonnegativity of $\{\wh f_j\}$,
\[
    \Big(\sum_{i,j=1}^N\wh f_i*\wh f_j\Big)^2
    \geq\Big(\sum_KH_K\Big)^2
    \geq\sum_KH_K^2.
\] 

Since there are $\asymp1+\log N$ dyadic integers $K$ with
$1\leq K\leq N/2$, Plancherel's theorem and \eqref{cone-example-one-scale} give
\begin{equation}
\label{cone-example-numerator}
\begin{aligned}
    \Big\|\sum_{j=1}^Nf_j\Big\|_4^4
    &=\Big\|\sum_{i,j=1}^N\wh f_i*\wh f_j\Big\|_2^2\\
    &\geq\sum_K\|H_K\|_2^2
      \gg(1+\log N)\delta^8.
\end{aligned}
\end{equation} 
Combining \eqref{cone-example-denominator} and
\eqref{cone-example-numerator}, and taking fourth roots, proves \eqref{cone-sharpness-conclusion}.
\end{proof}

\medskip 

\begin{remark}
\label{slow-growing-rmk}

\rm

The lower bound also shows that a uniform constant is impossible in Theorem~\ref{quartic-ortho-thm}, even with the additional cancellation \eqref{another-super-ortho}. 
Indeed, the example splits into $11$ residue classes satisfying both cancellation hypotheses.
Uniform bounds for these classes would recombine to contradict Proposition~\ref{cone-sharpness-prop}. 
This does not show that the exponent $2$ in \eqref{reverse-sqfcn} for $(1+\log N)^2$ is optimal.
\end{remark}

\bigskip

\section{Proof of Theorem \ref{thm:parabola}}
\label{sec:parabola}

We now give the  parabola argument, following the interlacing method of \cite{Cushman-Demeter-Wu}. 
All norms in this section are over $\ZR^2$, unless another space is specified. 
We keep the notation $\gamma$, $\ell$ and $\mathcal N_\rho$ from Subsection~\ref{sec:parabola-statement}.

Whenever the argument is applied to a subfamily, ignore its zero functions in the ordering, while keeping them equal to zero in the original family. 
Relabel the remaining pairs $(I_i,F_i)$ from left to right as $(I_1,F_1),\ldots,(I_m,F_m)$.  
Keep the original value of $\ell$.  
A block below is consecutive in this ordered list.  
Since $m\le N$, it is enough to prove a bound with loss $m^\varepsilon$.  
The induction is on $m$.

\medskip 

\subsection{Elementary results}

\begin{lemma}[Bounded Fourier overlap]\label{lem:overlap}
Let $u_\alpha\in L^2(\R^d)$ and suppose
$\supp(\widehat u_\alpha)\subset E_\alpha$.  If
\[
 \sup_{\zeta\in\R^d}\#\{\alpha:\zeta\in E_\alpha\}\le M,
\]
then
\[
 \Big\|\sum_\alpha u_\alpha\Big\|_2^2
 \le M\sum_\alpha\|u_\alpha\|_2^2.
\]
\end{lemma}

\begin{proof}
At each frequency, Cauchy--Schwarz gives
\[
 \Big|\sum_\alpha\widehat u_\alpha(\zeta)\Big|^2
 \le \Big(\sum_\alpha\mathbf{1}_{E_\alpha}(\zeta)\Big)
      \sum_\alpha|\widehat u_\alpha(\zeta)|^2.
\]
Integrate and use Plancherel.
\end{proof}

More generally, suppose that the product supports
$\supp(\widehat{G_{B_1}\cdots G_{B_k}})$, indexed by
$(B_1,\ldots,B_k)$, have overlap at most $M$.  Applying the lemma to these
products gives
\begin{equation}\label{eq:parabola-product-overlap}
 \Big\|\sum_B G_B\Big\|_{2k}^{2k}
 \le M\sum_{B_1,\ldots,B_k}
       \|G_{B_1}\cdots G_{B_k}\|_2^2
 =M\int\Big(\sum_B|G_B|^2\Big)^k.
\end{equation} 

\medskip 

We will use the following elementary estimate.

\begin{lemma}[Near-interaction graph]\label{lem:near-graph}
Let $\mathcal G$ be a graph of maximum degree at most $D$, with loops allowed,
and let $f_v$ be complex-valued functions.  Then
\[
 \Big|\sum_{(v,w)\in\mathcal G}f_v\overline{f_w}\Big|
 \le C D\sum_v|f_v|^2,
\]
where ordered or unordered edge conventions change only the absolute
constant.
\end{lemma}

\begin{proof}
Apply $2|f_vf_w|\le |f_v|^2+|f_w|^2$ to each edge and count incidences.
\end{proof}

\medskip 

\subsection{Sector recombination in \texorpdfstring{$L^3$}{L3}}

\begin{lemma}[Product-space interpolation]\label{lem:vector-interp}
Let $P_1,\ldots,P_h$ be Fourier projections on $\R^d$.  Assume that their
frequency regions have overlap at most $M_0$ almost everywhere and that,
for some $p>2$,
\[
 \|P_jf\|_p\le C_p\|f\|_p
 \qquad (1\le j\le h),
\]
with $C_p$ independent of $j$.  Define
$T(f_1,\ldots,f_h)=\sum_jP_jf_j$.  If $2<q<p$ and
\[
 \frac1q=\frac{1-\theta}{2}+\frac\theta p,
\]
then
\begin{equation}\label{eq:vector-interp}
 \|T(f_j)\|_q
 \ll_{p,M_0}h^{\theta(1-1/p)}
 \Big(\sum_{j=1}^h\|f_j\|_q^q\Big)^{1/q}.
\end{equation}
\end{lemma}

\begin{proof}
Finite overlap and Plancherel give
\[
 T:L^2(\R^d;\ell_h^2)\longrightarrow L^2(\R^d)
 \quad\text{with norm }O(M_0^{1/2}).
\]
The triangle inequality, the uniform $L^p$ projection bound, and H\"older give
\[
 T:L^p(\R^d;\ell_h^p)\longrightarrow L^p(\R^d)
 \quad\text{with norm }O(C_ph^{1-1/p}).
\]
Interpolate these two scalar operator bounds after identifying a sequence
$(f_j)$ with a function on $\R^d\times\{1,\ldots,h\}$ equipped with counting
measure.
\end{proof}

We need only the following planar consequence.  A planar sector is an
intersection of two half-planes.  Its sharp Fourier projection is a
composition of rotated one-dimensional Riesz projections and is uniformly
bounded on $L^p$, independently of direction and aperture.  A two-sided
sector is the union of a sector and its negative.

\begin{corollary}[Flat $\ell^3L^3$ sector estimate]\label{cor:flat-cone}
Suppose $U_1,\ldots,U_h\in\mathcal S(\R^2)$ have Fourier supports in
two-sided planar sectors of fixed bounded overlap away from the origin.
For every $\eta>0$,
\begin{equation}\label{eq:flat-cone}
 \Big\|\sum_{j=1}^hU_j\Big\|_3^3
 \ll_\eta h^{1+\eta}\sum_{j=1}^h\|U_j\|_3^3,
\end{equation}
uniformly in the directions and apertures.
\end{corollary}

\begin{proof}
Apply Lemma~\ref{lem:vector-interp} with $q=3$.  Here
$\theta=p/[3(p-2)]$, and after cubing the power of $h$ is
\[
 3\theta(1-1/p)=\frac{p-1}{p-2}=1+\frac1{p-2}.
\]
Choose $p$ sufficiently large in terms of $\eta$.
This planar estimate is also stated in \cite[Corollary 2.10]{Cushman-Demeter-Wu}.
\end{proof}

\medskip 

\subsection{Interlacing and thickened triple sums}

\begin{lemma}[Interlacing]\label{lem:interlacing}
Let
\[
 x\le y\le z,
 \qquad
 x'\le y'\le z',
\]
and suppose
\[
 x+y+z=x'+y'+z',
 \qquad
 x^2+y^2+z^2=x'^2+y'^2+z'^2.
\]
If the two multisets are distinct and $x<x'$, then
\begin{equation}\label{eq:interlace-chain}
 x<x'\le y'\le y\le z\le z'.
\end{equation}
If in addition $x<y<z$ and $x'<y'<z'$, then every inequality is strict:
$x<x'<y'<y<z<z'$.
\end{lemma}

\begin{proof}
Equality of the first two power sums gives equality of the first two
elementary symmetric functions.  Hence the monic cubics
\[
 P(T)=(T-x)(T-y)(T-z),
 \qquad
 Q(T)=(T-x')(T-y')(T-z')
\]
differ by a constant.  Since $x<x'\le y'\le z'$, one has $Q(x)<0$, and hence
$C=P(x)-Q(x)=-Q(x)>0$.  The roots of $Q$ are therefore the three intersections of the
graph of $P$ with the horizontal line of height $C$.  These intersections can
occur only where $P>0$, namely in $(x,y)\cup(z,\infty)$.  The polynomial $P$
is strictly increasing on $(z,\infty)$, so exactly one intersection lies
there.  The other two lie in $(x,y)$.  This is the strict form of
\eqref{eq:interlace-chain}; repeated roots follow by a limiting argument.
\end{proof}

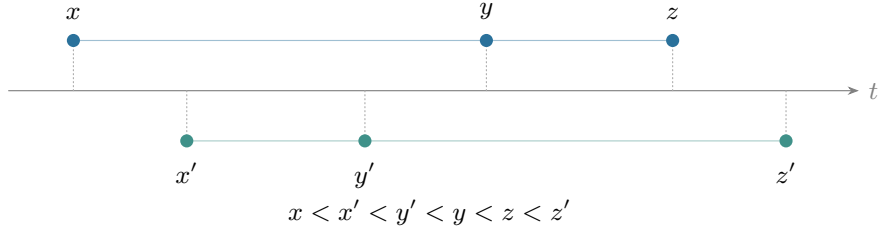
\begin{figure}[htbp]
\centering
\resizebox{.92\textwidth}{!}{%
\begin{tikzpicture}[x=1cm,y=1cm,>=Stealth,
  line cap=round,line join=round,
  parlabel/.style={font=\small}]
  \draw[axis] (-.2,0)--(10.3,0) node[right] {$t$};
  \draw[capblue!45,thin] (.6,.62)--(8,.62);
  \draw[capteal!45,thin] (2,-.62)--(9.4,-.62);
  \foreach \parpos/\partext in {.6/x,5.7/y,8/z}{
    \draw[gray!65,densely dotted] (\parpos,0)--(\parpos,.62);
    \fill[capblue] (\parpos,.62) circle (2.3pt);
    \node[parlabel,above=4pt] at (\parpos,.62) {$\partext$};
  }
  \foreach \parpos/\partext in {2/x',4.2/y',9.4/z'}{
    \draw[gray!65,densely dotted] (\parpos,0)--(\parpos,-.62);
    \fill[capteal] (\parpos,-.62) circle (2.3pt);
    \node[parlabel,below=4pt] at (\parpos,-.62) {$\partext$};
  }
  \node[parlabel] at (5,-1.5) {$x<x'<y'<y<z<z'$};
\end{tikzpicture}}
\caption{Two distinct sorted triples with the same first two
moments interlace. The strict ordering shown here is the case of
simple roots.}
\label{fig:interlacing}
\end{figure}

\begin{corollary}[Exact overlap via interlacing]\label{cor:exact-triple-overlap}
Let $J_1<\cdots<J_h$ be ordered intervals.  The family
\[
 \gamma(J_a)+\gamma(J_b)+\gamma(J_c),
 \qquad a\le b\le c,
\]
has overlap multiplicity $O(h)$.
\end{corollary}

\begin{proof}
The reader should first realize that if each $J$ consisted of just one point, this would trivially follow from basic strict convexity. Indeed, the choice of the point $J_a$ within a fiber would uniquely determine the other two points $J_b,J_c$. The proof for intervals relies critically on interlacing.

Fix a point $P=(P_1,P_2)\in\mathbb{R}^2$. A triple of intervals
$(J_a,J_b,J_c)$, with $a\le b\le c$, contributes to $P$ if
\[
    P\in\gamma(J_a)+\gamma(J_b)+\gamma(J_c).
\]
This means that there are parameters
\[
    x\in J_a,\qquad y\in J_b,\qquad z\in J_c
\]
such that
\[
    P=\gamma(x)+\gamma(y)+\gamma(z),
\]
or, equivalently,
\[
    x+y+z=P_1,\qquad x^2+y^2+z^2=P_2.
\]
Since the intervals are ordered, we may arrange that $x\le y\le z$,
sorting the parameters within any interval that occurs more than once.
We must show that only $O(h)$ distinct triples of intervals can
contribute to this fixed $P$.

First, we use the convention that a shared endpoint belongs only to the
interval on its right, so that every parameter belongs to a unique
interval. For each contributing triple of intervals, choose one triple
of parameters $(x,y,z)$ as above.

Two chosen triples with the same smallest parameter $x$ must coincide.
Indeed,
\[
    y+z=P_1-x,\qquad y^2+z^2=P_2-x^2
\]
determine the ordered pair $y\le z$ uniquely. Since each parameter
belongs to a unique interval, coinciding parameter triples come from
the same triple of intervals.

We can therefore list the distinct contributing triples as
\[
    (J_{a_1},J_{b_1},J_{c_1}),\ldots,
    (J_{a_m},J_{b_m},J_{c_m}),
\]
with chosen parameters
\[
    x_\nu\in J_{a_\nu},\qquad
    y_\nu\in J_{b_\nu},\qquad
    z_\nu\in J_{c_\nu},
\]
in such a way that
\[
    x_1<x_2<\cdots<x_m.
\]

Consider two successive triples. Their parameters satisfy
\[
\begin{aligned}
    x_\nu+y_\nu+z_\nu
    &=x_{\nu+1}+y_{\nu+1}+z_{\nu+1}=P_1,\\
    x_\nu^2+y_\nu^2+z_\nu^2
    &=x_{\nu+1}^2+y_{\nu+1}^2+z_{\nu+1}^2=P_2.
\end{aligned}
\]
Applying Lemma~5.5 with
\[
\begin{aligned}
    (x,y,z)&=(x_\nu,y_\nu,z_\nu),\\
    (x',y',z')&=(x_{\nu+1},y_{\nu+1},z_{\nu+1}),
\end{aligned}
\]
we obtain
\[
    x_\nu<x_{\nu+1}
    \le y_{\nu+1}\le y_\nu
    \le z_\nu\le z_{\nu+1}.
\]
Thus the smallest and largest parameters move to the right, while the
middle parameter moves to the left. Because $J_1,\ldots,J_h$ are
ordered, their interval indices satisfy
\[
    a_\nu\le a_{\nu+1},\qquad
    b_\nu\ge b_{\nu+1},\qquad
    c_\nu\le c_{\nu+1}.
\]

At least one of these three indices changes at every step, since the
contributing triples of intervals are distinct. Each index takes
values in $\{1,\ldots,h\}$ and moves in only one direction, so each can
change at most $h-1$ times. Consequently,
\[
    m-1\le 3(h-1),
    \qquad\text{and hence}\qquad
    m\le 3h-2.
\]

Finally, allowing a shared endpoint to belong to both adjacent
intervals changes the count by at most an absolute factor. Under the
convention above, a parameter assigned to $J_j$ could originally have
belonged to $J_j$ or $J_{j-1}$. Thus each contributing triple under the
unique-endpoint convention accounts for at most $2^3$ original triples.
The number of contributing triples is therefore still $O(h)$, uniformly
in $P$.

\end{proof}

We next show that approximate equality of the first two moments can be changed
into exact equality by moving the parameters only a little.  For
$\mathbf{t}=(t_1,t_2,t_3)$, set
\[
 s(\mathbf{t})=\sum_{i=1}^3t_i,
 \quad q(\mathbf{t})=\sum_{i=1}^3t_i^2,
 \quad \mu(\mathbf{t})=\frac{s(\mathbf{t})}3,
 \quad r(\mathbf{t})^2=q(\mathbf{t})-\frac{s(\mathbf{t})^2}{3}.
\]
For fixed $s$ and $q$, the possible triples lie on a circle, possibly reduced
to one point, in the plane $t_1+t_2+t_3=s$.

\begin{lemma}[Correcting approximate moment equalities]\label{lem:exactification}
Let $\mathbf{t},\mathbf{u}\in[0,1]^3$ be sorted and suppose
\[
 |s(\mathbf{t})-s(\mathbf{u})|+|q(\mathbf{t})-q(\mathbf{u})|\le C A\ell^2.
\]
There is a sorted triple $\widetilde{\mathbf{u}}\in\R^3$ with exactly the same
first two moments as $\mathbf{t}$ and
\[
 \max_i|\widetilde u_i-u_i|\le C_A\ell.
\]
\end{lemma}

\begin{proof}
We have
\[
 |\mu(\mathbf{t})-\mu(\mathbf{u})|\ll A\ell^2,
 \qquad
 |r(\mathbf{t})^2-r(\mathbf{u})^2|\ll A\ell^2,
\]
and hence $|r(\mathbf{t})-r(\mathbf{u})|\ll_A\ell$.  If
$r(\mathbf{u})>0$, define
\[
 \widetilde{\mathbf{u}}
 =\mu(\mathbf{t})(1,1,1)
  +\frac{r(\mathbf{t})}{r(\mathbf{u})}
       \big(\mathbf{u}-\mu(\mathbf{u})(1,1,1)\big).
\]
Translation and nonnegative radial scaling preserve coordinate order.  Moreover,
\[
 |\widetilde u_i-u_i|
 \le |\mu(\mathbf{t})-\mu(\mathbf{u})|
 +|r(\mathbf{t})-r(\mathbf{u})|
   \frac{|u_i-\mu(\mathbf{u})|}{r(\mathbf{u})}
 \ll_A\ell.
\]
If $r(\mathbf{u})=0$, take $\widetilde{\mathbf{u}}=\mathbf{t}$.  Since every $u_i$ equals
$\mu(\mathbf{u})$,
\[
    |t_i-u_i|\le |t_i-\mu(\mathbf{t})|+|\mu(\mathbf{t})-\mu(\mathbf{u})| \le r(\mathbf{t})+O(A\ell^2)\ll_A\ell. \qedhere
\]
\end{proof}

The corrected triple may lie slightly outside $[0,1]$.  This is harmless in
the next auxiliary argument: we extend $\gamma(t)=(t,t^2)$ algebraically to
all real $t$ and use unbounded exterior cells.

\begin{proposition}[Overlap of thickened triple sums]\label{prop:robust-triple-overlap}
Let $K_1<\cdots<K_h$ be the hulls of consecutive blocks in a relabeled
nonzero subfamily.  Then
\begin{equation}\label{eq:robust-triple-overlap}
 \operatorname{mult}\Big\{
 \mathcal N_{3A\ell^2}
   \big(\gamma(K_a)+\gamma(K_b)+\gamma(K_c)\big):
 a\le b\le c\Big\}
 \le C_Ah.
\end{equation}
\end{proposition}

\begin{proof}
Put boundaries at the midpoints of the gaps between successive hulls and let
$J_j$ be the cell containing $K_j$; take the two exterior cells unbounded.
Every finite cell has length at least $\ell$.  Fix a frequency point belonging
to several sets in \eqref{eq:robust-triple-overlap} and choose one approximate
representation as a reference.  Every other representation differs from the
reference by $O(A\ell^2)$ in the first two moments.  By
Lemma~\ref{lem:exactification}, move its three parameters by $O_A(\ell)$ to the
exact reference set.  Each coordinate crosses only $O_A(1)$ cells.  The
exact cell labels have multiplicity $O(h)$ by
Corollary~\ref{cor:exact-triple-overlap}, and each has only $O_A(1)$ possible nearby
label triples.
\end{proof}

\medskip 

\subsection{Fourier supports of separated pairs}

Fix a large integer $L=L(A)$ with $L^2\ge C_0A$.  For a consecutive block
$M_j$, call $i,i'\in M_j$ near if $|i-i'|\le L$ and far otherwise.  Put
\[
 G_j=\sum_{i\in M_j}F_i
\]
and split
\begin{align}
 D_j&=\sum_{i\in M_j}|F_i|^2
 +2\operatorname{Re}
  \sum_{\substack{i<i',\ i,i'\in M_j\\
        |i-i'|\le L}}
       F_i\overline{F_{i'}},\label{eq:parabola-D}\\
 O_j&=|G_j|^2-D_j.\label{eq:parabola-O}
\end{align}
By Lemma~\ref{lem:near-graph},
\begin{equation}\label{eq:D-pointwise}
 |D_j|\ll_L\sum_{i\in M_j}|F_i|^2.
\end{equation}

For $u<v$,
\begin{equation}\label{eq:chord-direction}
 \gamma(v)-\gamma(u)=(v-u)(1,u+v).
\end{equation}
The slope of this chord is determined by the sum of its endpoints.

\begin{lemma}[Separated pairs lie in disjoint sectors]\label{lem:chord-sectors}
For every nonzero $O_j$ there is a two-sided planar sector $\mathcal C_j$
such that
\[
 \supp(\widehat O_j)
 \subset \mathcal C_j\cap\{\xi:|\xi_1|\ge cL\ell\}.
\]
The angular parts of the $\mathcal C_j$ are pairwise disjoint.  Although the
closed sectors meet at the origin, the Fourier supports stay a distance
$\gg L\ell$ from the origin and are therefore disjoint.
\end{lemma}

\begin{proof}
The explicit expansion
\[
 O_j=2\operatorname{Re}
 \sum_{\substack{i<i',\ i,i'\in M_j\\
       |i-i'|>L}}
 F_i\overline{F_{i'}}
\]
is essential.  Every diagonal or near term whose support might approach
the origin belongs to $D_j$.  A frequency in a retained difference support
has the form, for some $u\in I_i$ and $v\in I_{i'}$ with $i<i'$,
\[
 \xi=\pm d(1,s)+e,
 \qquad d=v-u,\quad s=u+v,\quad |e|\ll A\ell^2.
\]
For a far pair $i<i'$, one has $i'-i\ge L+1$, so at least $L$ intervals of the
current subfamily lie strictly between $I_i$ and $I_{i'}$.  Their total length
is at least $L\ell$, and therefore $d=v-u\ge L\ell$.  Since $d\le1$ whenever a
retained pair exists, $\ell\ll L^{-1}$, and
\[
 \frac{|e|}{d}\ll\frac{A\ell}{L}\ll\frac{A}{L^2}.
\]
Choosing $C_0$ large gives $|\xi_1|\ge d/2\gg L\ell$ and
\[
 \Big|\frac{\xi_2}{\xi_1}-s\Big|
 \ll\frac{A\ell}{L}.
\]

Let $b_j$ be the right endpoint of the hull of $M_j$, and let $a_{j+1}$ be
the left endpoint of the next hull.  Far chord slopes in $M_j$ are at most
$2b_j-cL\ell$, while those in $M_{j+1}$ are at least
$2a_{j+1}+cL\ell\ge2b_j+cL\ell$.  The exact slope gap is therefore
$\gg L\ell$, whereas thickening changes slope by
$O(A\ell/L)\ll L\ell$.  The positive sectors are disjoint.  The conjugate
terms lie in the corresponding negative sectors.
\end{proof}

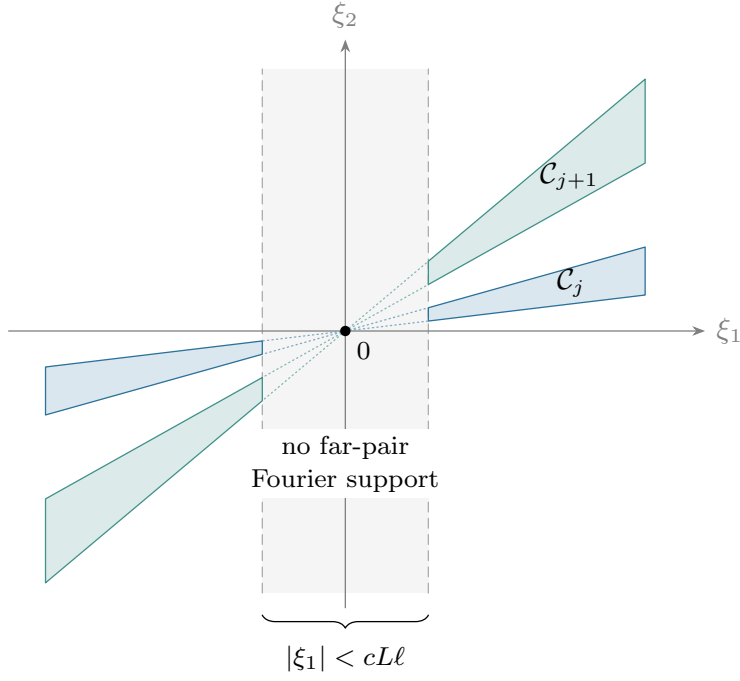
\begin{figure}[htbp]
\centering
\resizebox{.78\textwidth}{!}{%
\begin{tikzpicture}[x=1cm,y=.80cm,>=Stealth,
  line cap=round,line join=round,
  parlabel/.style={font=\small},
  partiny/.style={font=\footnotesize}]
  \fill[gray!8] (-.9,-3.55) rectangle (.9,3.55);
  \draw[gray!65,densely dashed] (-.9,-3.55)--(-.9,3.55);
  \draw[gray!65,densely dashed] (.9,-3.55)--(.9,3.55);
  \foreach \parsign in {-1,1}{
    \begin{scope}[xscale=\parsign,yscale=\parsign]
      \path[fill=capblue!17,draw=capblue]
        (.9,.135)--(3.25,.4875)--(3.25,1.1375)--(.9,.315)--cycle;
      \path[fill=capteal!18,draw=capteal]
        (.9,.63)--(3.25,2.275)--(3.25,3.4125)--(.9,.945)--cycle;
      \draw[capblue!60,densely dotted] (0,0)--(.9,.135);
      \draw[capblue!60,densely dotted] (0,0)--(.9,.315);
      \draw[capteal!65,densely dotted] (0,0)--(.9,.63);
      \draw[capteal!65,densely dotted] (0,0)--(.9,.945);
    \end{scope}
  }
  \draw[axis] (-3.65,0)--(3.9,0) node[right] {$\xi_1$};
  \draw[axis] (0,-3.75)--(0,3.95) node[above] {$\xi_2$};
  \fill (0,0) circle (1.6pt);
  \node[partiny,below right] at (0,0) {$0$};
  \node[parlabel] at (2.45,.65) {$\mathcal C_j$};
  \node[parlabel] at (2.43,2.09) {$\mathcal C_{j+1}$};
  \node[partiny,align=center,fill=white,inner sep=2pt] at (0,-1.8)
    {no far-pair\\Fourier support};
  \draw[decorate,decoration={brace,mirror,amplitude=4pt}]
    (-.9,-3.85)--(.9,-3.85)
    node[midway,below=6pt,partiny] {$|\xi_1|<cL\ell$};
\end{tikzpicture}}
\caption{Far differences from different blocks lie in sectors with
disjoint angular parts; the
Fourier supports satisfy $|\xi_1|\geq cL\ell$, so the common vertex of
the sectors is absent.}
\label{fig:chord-sectors}
\end{figure}

\medskip 

\subsection{The recurrence}

We first assume that the nonzero norms
$a_i=\|F_i\|_6$ lie within a factor of two.  Let $B(m)$ be the best constant
in \eqref{eq:parabola-main}, without a prescribed $m^\varepsilon$ factor, over
families with at most $m$ nonzero summands of comparable mass.  Omitted
functions are left equal to zero, so every further subfamily is an admissible
instance of the same estimate.  In particular, $B(m)$ is nondecreasing.

Fix a large integer $h$.  Assume first that $m\ge h$; the range $m<h$ will be
included in the base case.  Divide $\{1,\ldots,m\}$ into $h$ nonempty
consecutive blocks $M_1,\ldots,M_h$ whose cardinalities differ by at most one.
Set
\[
 S=\sum_i a_i^2,
 \qquad s_j=\sum_{i\in M_j}a_i^2.
\]
Comparability and balanced cardinality imply
\begin{equation}\label{eq:mass-balance}
 \max_js_j\ll h^{-1}S,
 \qquad
 \sum_{j=1}^h s_j^3\ll h^{-2}S^3.
\end{equation}

Apply Lemma~\ref{lem:overlap} to the products
$G_{j_1}G_{j_2}G_{j_3}$ and use
Proposition~\ref{prop:robust-triple-overlap}.  This gives
\begin{equation}\label{eq:parabola-master}
 \Big\|\sum_iF_i\Big\|_6^6
 \ll_A h\int_{\R^2}\Big(\sum_{j=1}^h|G_j|^2\Big)^3.
\end{equation}
Indeed, after sorting the three block labels, each sorted triple has at most
six ordered permutations.  Thus the product supports have multiplicity
$O_A(h)$.

Write $D=\sum_jD_j$ and $O=\sum_jO_j$.  Then
$\sum_j|G_j|^2=D+O$.  Since $D+O\ge0$,
\[
 \int(D+O)^3=\|D+O\|_3^3
 \ll \|D\|_3^3+\|O\|_3^3.
\]
From \eqref{eq:D-pointwise} and Minkowski,
\begin{equation}\label{eq:parabola-near}
 \|D\|_3\ll_L\sum_i\||F_i|^2\|_3
 \ll_L S.
\end{equation}
The far pieces lie in the disjoint sectors of
Lemma~\ref{lem:chord-sectors} and stay away from the origin.  Therefore
Corollary~\ref{cor:flat-cone} gives
\[
 \|O\|_3^3\ll_\eta h^{1+\eta}\sum_j\|O_j\|_3^3.
\]
If $m_j=|M_j|\le\lceil m/h\rceil$, then
\begin{align*}
 \|O_j\|_3
 &\le \||G_j|^2\|_3+\|D_j\|_3\\
 &\ll_L\big(B(\lceil m/h\rceil)^2+1\big)s_j.
\end{align*}
Together with \eqref{eq:mass-balance}, this yields
\begin{equation}\label{eq:parabola-far}
 \|O\|_3^3
 \ll_{A,\eta}h^{-1+\eta}
 \big(B(\lceil m/h\rceil)^6+1\big)S^3.
\end{equation}
Substitution of \eqref{eq:parabola-near} and
\eqref{eq:parabola-far} into \eqref{eq:parabola-master} gives
\begin{equation}\label{eq:parabola-recurrence}
 B(m)^6\le C_{A,\eta}
 \Big[h+h^\eta\big(B(\lceil m/h\rceil)^6+1\big)\Big].
\end{equation}

For $m\ge h$, $\lceil m/h\rceil\le2m/h$.  Assuming
$B(r)^6\le K r^{6\varepsilon}$ for $r<m$, the recursive contribution is at most
\[
 C_{A,\eta}K2^{6\varepsilon}h^{\eta-6\varepsilon}m^{6\varepsilon}.
\]
Take $\eta=3\varepsilon$ and then choose $h=h(A,\varepsilon)$ so that
$C_{A,\eta}2^{6\varepsilon}h^{-3\varepsilon}\le\tfrac12$.  After enlarging $K$ to absorb
$C_{A,\eta}(h+h^\eta)$ and all $m<h$, strong induction gives
$B(m)^6\le K m^{6\varepsilon}$.  Thus
\begin{equation}\label{eq:B-bound}
 B(m)\le C_{A,\varepsilon}m^\varepsilon
\end{equation}
for comparable masses.

To remove comparability, let
\[
 \mathsf S=\Big(\sum_i a_i^2\Big)^{1/2},
 \qquad K=\lceil3\log_2N\rceil,
\]
and form the classes
\[
 \mathcal C_k=\{i:2^{-k-1}\mathsf S<a_i\le2^{-k}\mathsf S\},
 \qquad 0\le k\le K.
\]
Put
\[
 \mathsf S_k=\Big(\sum_{i\in\mathcal C_k}a_i^2\Big)^{1/2}.
\]
The remaining tail has $L^6$ norm at most
$N2^{-K-1}\mathsf S\le N^{-2}\mathsf S$.  Apply
\eqref{eq:B-bound} with exponent $\varepsilon/2$ to every nonempty class.  Then
\begin{align*}
 \Big\|\sum_iF_i\Big\|_6
 &\le C_{A,\varepsilon}N^{\varepsilon/2}\sum_{k=0}^K\mathsf S_k
      +N^{-2}\mathsf S\\
 &\le C_{A,\varepsilon}N^{\varepsilon/2}\sqrt{K+1}
      \Big(\sum_{k=0}^K\mathsf S_k^2\Big)^{1/2}+N^{-2}\mathsf S\\
 &\le C_{A,\varepsilon}N^\varepsilon\mathsf S.
\end{align*}
This proves Theorem~\ref{thm:parabola}.

\medskip 

\subsection{The energy consequence}\label{sec:parabola-energy}

We make explicit how the decoupling theorem recovers the parabola case of
\cite[Theorem~1.1]{Cushman-Demeter-Wu}. The argument is standard. For a finite nonempty set $X\subset\R$ and complex
coefficients $c=(c_t)_{t\in X}$, put $\gamma(t)=(t,t^2)$ and define
\[
 \mathcal E_3(c;X)=
 \sum_{\zeta\in\gamma(X)+\gamma(X)+\gamma(X)}
 \left|
 \sum_{\substack{t_1,t_2,t_3\in X\\
       \gamma(t_1)+\gamma(t_2)+\gamma(t_3)=\zeta}}
       c_{t_1}c_{t_2}c_{t_3}
 \right|^2.
\]
Theorem~\ref{thm:parabola} implies, for every $\varepsilon>0$,
\[
 \mathcal E_3(c;X)^{1/6}
 \le C_\varepsilon |X|^\varepsilon
       \left(\sum_{t\in X}|c_t|^2\right)^{1/2}.
\]
In particular, taking $c_t=1$ and renaming the exponent gives
\[
 J_3(\gamma(X))
 :=\#\left\{(t_1,\ldots,t_6)\in X^6:
       \sum_{i=1}^3\gamma(t_i)=\sum_{i=4}^6\gamma(t_i)\right\}
 \le C_\varepsilon |X|^{3+\varepsilon}.
\]

\begin{proof}[Derivation from Theorem~\ref{thm:parabola}]
An affine change of the parameter preserves equalities of the first two
moments, so assume $X\subset(0,1)$. Choose a partition of $[0,1]$ into
$|X|$ intervals with one point of $X$ in the interior of each, and let $\ell$
be its minimum width. The distinct triple sums of $\gamma(X)$ form a finite
set. Choose a nonzero Schwartz function $\phi$ whose Fourier support is a
sufficiently small ball about zero that the functions
\[
 F_t(x)=c_t\phi(x)e^{2\pi\mathrm{i} x\cdot\gamma(t)}
\]
satisfy the support hypotheses of Theorem~\ref{thm:parabola}, with $A=1$,
and the translates of $\supp(\widehat{\phi^3})$ by distinct triple sums are
disjoint. Cubing the sum, grouping equal triple sums, and applying Plancherel
then gives the exact identity
\[
 \left\|\sum_{t\in X}F_t\right\|_6^6
 =\|\phi\|_6^6\,\mathcal E_3(c;X).
\]
Since $\|F_t\|_6=|c_t|\|\phi\|_6$, Theorem~\ref{thm:parabola} and cancellation
of $\|\phi\|_6$ give the claimed weighted bound. Although the support radius
of $\widehat\phi$ depends on $X$, the decoupling constant does not. Thus no
separation or diameter parameter enters the conclusion.
\end{proof}

\bigskip

\bibliographystyle{alpha}
\bibliography{bibli}

\end{document}